\documentclass[letterpaper, 10 pt, conference]{ieeeconf}  % Comment this line out if you need a4paper

\usepackage{cite}
\usepackage{amsmath,amssymb,amsfonts}
\usepackage{algorithm}
\usepackage{algorithmic}
\usepackage{graphicx}
\usepackage{textcomp}
\usepackage{float}
\usepackage{caption}
\usepackage{subcaption}
\usepackage{url}
\usepackage{booktabs} % for professional tables

\usepackage{xspace}
\usepackage{soul}
\usepackage{mathtools}
\usepackage{makecell}
\usepackage{nicefrac}

\newcommand{\la}{\leftarrow}

\newcommand{\ie}{\unskip, i.\,e.,\xspace}
\newcommand{\eg}{\unskip, e.\,g.,\xspace}

\newcommand{\N}{\ensuremath{\mathbb{N}}}

\newcommand{\X}{\ensuremath{\mathbb{X}}}

\newcommand{\U}{\ensuremath{\mathbb{U}}}
\newcommand{\spc}{\ensuremath{\,\,}}

\DeclareMathOperator*{\argmin}{arg\,min}

\newtheorem{lemma}{Lemma}

\title{Effects of sampling and horizon in predictive reinforcement learning}

\author{Pavel Osinenko, Dmitrii Dobriborsci% <-this % stops a space
\thanks{The authors are with the Computational and Data Science and Engineering Center, Skolkovo Institute of Science and Technology.
}}

\begin{document}

\maketitle
\thispagestyle{empty}
\pagestyle{empty}

%%%%%%%%%%%%%%%%%%%%%%%%%%%%%%%%%%%%%%%%%%%%%%%%%%%%%%%%%%%%%%%%%%%%%%%%%%%%%%%%
\begin{abstract}
Plain reinforcement learning (RL) may be prone to loss of convergence, constraint violation, unexpected performance, etc.
Commonly, RL agents undergo extensive learning stages to achieve acceptable functionality.
This is in contrast to classical control algorithms which are typically model-based.
An direction of research is the fusion of RL with such algorithms, especially model-predictive control (MPC).
This, however, introduces new hyper-parameters related to the prediction horizon.
Furthermore, RL is usually concerned with Markov decision processes.
But the most of the real environments are not time-discrete.
The factual physical setting of RL consists of a digital agent and a time-continuous dynamical system.
There is thus, in fact, yet another hyper-parameter -- the agent sampling time.
In this paper, we investigate the effects of prediction horizon and sampling of two hybrid RL-MPC-agents in a case study with a mobile robot parking, which is in turn a canonical control problem.
We benchmark the agents with a simple variant of MPC.
The sampling showed a kind of a ``sweet spot'' behavior, whereas the RL agents demonstrated merits at shorter horizons.
%Remarkably, 
%This paper is concerned with a canonical control example of a mobile robot parking problem where two variations of RL, that employ prediction, namely, a roll-out and so-called stacked Q-learning, are compared to each other and to the benchmark -- a model-predictive controller.
%The environment is modeled in a hybrid setting: the system dynamics are continuous, whereas the agent (the controller) is digital, which corresponds to the real physical conditions.
%Effects of controller discretization, prediction step size, and prediction horizon length are investigated, whereas some interesting tendencies are observed.
%Remarkably, the stacked RL approach demonstrated superior performance in some scenarios.
\end{abstract}

%\begin{IEEEkeywords}
%Stability
%\end{IEEEkeywords}

%\def\BibTeX{{\rm B\kern-.05em{\sc i\kern-.025em b}\kern-.08em
%    T\kern-.1667em\lower.7ex\hbox{E}\kern-.125emX}}
    
%\begin{document}
%
%\history{Date of publication xxxx 00, 0000, date of current version xxxx 00, 0000.}
%\doi{10.1109/ACCESS.2017.DOI}
%
%\title{Effects of Sampling and Prediction Horizon in Reinforcement Learning}
%\author{\uppercase{Pavel Osinenko},
%\uppercase{Dmitrii Dobriborsci, Ksenia Makarova,} \uppercase{and}
%\uppercase{Aleksandr Rubashevskii}}
%\address{Computational and Data Science and Engineering, Skolkovo Institute of Science and Technology, Russia}
%% \address[2]{Department of Physics, Colorado State University, Fort Collins, 
%% CO 80523 USA (e-mail: author@lamar.colostate.edu)}
%% \address[3]{Electrical Engineering Department, University of Colorado, Boulder, CO 
%% 80309 USA}
%\tfootnote{This work was supported by Skolkovo Institute of Science and Technology}
%
%\markboth
%{Osinenko P. \headeretal: submission to journal IEEE Access}
%{Osinenko P. \headeretal: submission to journal IEEE Access}

%\corresp{Corresponding author: Dmitrii Dobriborsci (e-mail: d.dobriborsci@skoltech.ru)}
%\begin{keywords}
%reinforcement learning, predictive control, simulation, mobile robot
%\end{keywords}

%\titlepgskip=-15pt
%
%\maketitle

\section{Introduction}
\label{sec:introduction}

\PARstart{R}{einforcement} Learning (RL) shows remarkable performance in playground settings of video- and table games such as Starcraft, chess and Go \cite{silver16, silver18, vinyals19}.
Industry-close applications appear more challenging to RL due to the lack of freedom in training \cite{LiCh19, krauth19, yang19, park20, Kakade20}.
This may be related to limited resources and technical constraints.
In general, tuning of RL algorithms is a sensitive matter \cite{Kober13, LiYi19}.
% At the same time, they have limited practical use due to the lack of built-in safety and stability mechanisms.
% \cite{sutton18}
Currently, industry is dominated by classical control-theoretic methods such as model predictive control (MPC) \cite{Soest06, Kouro09, Ma12}.
MPC is widely used in such areas as chemical industry and oil refining \cite{Qin03, Hrovat12, Darby12, Forbes15}.
%Reinforcement Learning is closely related to \blue{a} classical control theory \cite{Bertsekas05} which is well-known and widely used in industrial tasks \cite{Soest06, Kouro09, Ma12}.
%In contrast to RL, control algorithms are commonly model-based predictive.
%The most well-known family of control algorithms is Model Predictive Control (MPC), which is mainly used in the chemical industry and oil refining \cite{Qin03, Hrovat12, Darby12, Forbes15}.
Somewhat in contrast to the classical control, RL is aimed at a learning-based, model-free (in some configurations) approach.
Nevertheless, it is perhaps the model-based formal guarantees that make classical control attractive to the industry.
In fact, integration of predictive mechanisms into RL is not new (see \eg \cite{Bertsekas99,Bertsekas05} for a reward roll-out methodology).
%This work goes along the lines of fusion of RL with predictive controls and addresses the tuning of the latter.
%\blue{A} roll-out approach can be described as a multiple application of some base heuristic (i.e. \blue{a} recursive view) to \blue{an} algorithm and corresponding calculations of \blue{an} expected reward of the final form.
% Nonetheless, predictive mechanisms can also be integrated into RL through the modification of Q-learning algorithm \cite{Watkins92}, namely the roll-out \cite{Bertsekas99} and the stacked \cite{Bertsekas05} versions. 
% Automotive industry

\textbf{Related work.} Up to date, using predictive elements from a classical control theory to improve RL approaches is an active area of research.
For instance, the promising recent concept of a so-called RL Dreamer effectively uses image-based prediction reminiscent to adaptive MPC \cite{hafner20}.
%In \cite{hafner20} the overall scheme was an adaptive MPC with RL elements related to the cost.
Some model estimation techniques for prediction in the so-called representational learning were introduced in \cite{guo20}.
In the so-called differentiable MPC \cite{amos18}, a combination of model-free and model-based RL elements was suggested.
An off-policy actor-critic deep value MPC combining model-based trajectory optimization with the value function estimation was developed in \cite{hoeller20}.
Another attempt to combine model-based approach and learning techniques is described in \cite{East20} where a differentiable linear quadratic MPC framework for a safe imitation learning was proposed.
A predictive scheme was suggested in \cite{reddy20}, via an RL agent combined planning with MPC to learn a forward dynamics model.
The development of this direction has led to the results presented in \cite{karnchanachari20}.
The authors used MPC to learn a cost function from scratch via high-level objective learning and tested it on the real ground vehicle.
%\cite{Asis20} studied the sum of agent rewards over a fixed number of future time steps predicted via temporal difference with a fixed horizon size. 
Clearly, integration of prediction methods in RL is gaining attraction.
Therefore, questions of hyper-parameter effects is of relevance.
%Those works clearly demonstrate the relevance of ideas of \blue{the} fusion of RL with control.

% authors incorporate predictive control into RL

% receding-horizon control - \cite{Jadbabaie01}
% \cite{Lenz15} - deep latent features for MPC.
% \cite{Finn16} - used MPC.

Speaking of MPC, it has certain fundamental hyper-parameters, such as horizon length and prediction step size, impacting the overall performance \cite{Cannon05, Tippett14, Wojsznis03, Worthmann12, Sawma18}.
The current study focuses on the effects of such hyper-parameters, but also time discretization step size (in brief, sampling time).
Sampling time may have drastic effects on system stability \cite{Silva10, Zhang14, khosla89}.
%For instance, the work \cite{khosla89} demonstrated significant effects of sampling on the performance of the considered model-based controllers.
A stable and optimally tuned system with a continuous controller may be destabilized after time discretization \cite{Laskawski15}.
%The prediction step size and the prediction horizon length are attributes of predictive models when future states are approximated to achieve greater accuracy and stability.
%Questions of hyperparameter tuning in control are well-studied.
%Several works studied the dependence of the control settings on the sampling rate \cite{ Silva10, Zhang14}, the prediction step size \cite{Cannon05, Tippett14}, the prediction horizon length \cite{Wojsznis03, Worthmann12, Sawma18}.
% Only little work has merged the extensive knowledge on 
%There are not  papers that examine the effect of these hyperparameters on RL agents.
When it comes to RL, there is evidence of performance deterioration of Q-learning under high sampling rates \cite{tallec19}.
Authors suggested incorporating the advantage function to remedy issues of collapsing Q-functions.
% and demonstrates the relevance of classic control theory, namely sample-and-hold \cite{Osinenko2018-pract-stabilization}.
An adaptive discretization for model-based RL via an optimistic one-step value iteration approach was proposed in \cite{Sinclair20}. 
However, not much attention has been paid to the effects of prediction step size and prediction horizon length parameters in RL.
%The goal of this work is to address these effects, as well as the effects of sampling.

% So our goal is to expand this area of research and to provide a case study of the dependence of Q-learning-based algorithms \cite{Watkins92} (MPC, N-roll-out Q-learning \cite{Bertsekas99}, and stacked Q-learning \cite{Bertsekas05}) on main control-related hyperparameters: sampling rate, prediction step size, and prediction horizon size.

% suggests to employ (in general) imperfect parametrized approximations for the value function and simultaneuosly the control action [22]. Consequently, while evaluation of the controller yields 70 suboptimal performance, the parameters or weights are updated according to the so-called temporal differences [38, 39] to achieve subsequent lower costs.

% TD (Temporal Difference) % introduce TD?

% Bellman equation

% Prediction decision and control algorithm

% \subsection{Importance of prediction horizon}
% The prediction model is used to predict the future output based on historical information about the process, as well as anticipated future input. So, the size of such a prediction horizon is directly connected with the quality of prediction - more future predictions will be less accurate than the first ones. That is why it is very important to choose the magnitude of this predictive horizon to strike a balance between accuracy and the horizon of knowledge that we want to obtain.

% Purpose, approach and summary of findings
\textbf{Summary.} % Statement of contribution
In this study, the sensitivity of three predictive agents (MPC, a roll-out Q-learning, and the so-called stacked Q-learning) to the selected hyper-parameters (sampling rate, prediction step size and prediction horizon length) is investigated on a canonical example of a wheeled robot with dynamic steering torque and pulling force, also known as the extended nonholonomic double integrator \cite{sontag04, clarke10, kellet17, Abbasi17}.
Such a system is used in numerous studies for benchmarking \cite{Aguiar00, Pascoal02, Fazal-ur-Rehman05, Devon07, Watanabe10}.
Main findings can be summarized as follows.
In terms of the overall tendency for all the methods, too low sampling time leads to a performance deterioration (in terms of the accumulated stage cost), which may be explained by too short-sighted prediction.
Increasing the sampling time improves the performance to a certain point where prediction error starts to dominate.
%\red{A similar tendency was observed for the prediction step size.}
Increasing the horizon length boosts the effects of prediction in all the methods and generally leads to a better performance, although at a higher computational cost.
Speaking of the method comparison, remarkably, the stacked Q-learning tends to outperform both the MPC and the roll-out algorithms at shorter horizons.
In particular, the stacked Q-learning achieved more successful robot parking count.
It should be noted here that the roll-out and stacked Q-learning have similar computational complexity.  

\textbf{Notation.}
Sequences: for any $z$: $\{z_{i|k}\}_i^N = \{ z_{1|k}, \dots, \\ z_{N|k} \} = \{ z_k, \dots, z_{k+N-1} \}$, if the starting index $k$ is emphasized; otherwise, just $\{z_i\}^N = \{ z_1, \dots, z_N \}$.
If $N$ in the above is omitted, the sequence is considered infinite.

\section{Environment description}
\label{sampling}

% \subsection{General type of environment: hybrid system} % alternatives to sample-and-hold?

% The development of ideas for integrating control theory methods into reinforcement learning can provide formal guarantees of sustainability and safety. Methods that have already demonstrated potential that use elements of MPC and Lyapunov stability theory, including in the works \cite{Beckenbach18, Osinenko20}

% \begin{equation}
% % \mathcal{D}^{+} x = f(x, u)
% \dot{x} = f(x, u)
% \end{equation}

% But using sample-and-hold:
% \begin{equation}
% x_{k+1} = f(x_k, u_k)
% \end{equation}

There are three types of the closed-loop setup (agent-environment \ie controller-system), namely, pure discrete, pure continuous, and hybrid. 
A pure discrete-time design is a discrete system with a discretized controller \cite{Nevsic99}. 
A pure continuous-time is a continuous-time system design with a continuous controller, which is rarely possible unless the controller is analogue.
A hybrid system is a continuous-time system design with a discretized controller.
A continuous-time system design is more suitable for linear systems since some important structural properties might be lost after discretization.
The hybrid setting with a continuous-time environment and a discretized controller is considered as a more realistic as the other two variants, and is used as the foundation in this work.
% , and each has its strengths.
% A continuous-time environment is closer to the real physical objects, and a discrete-time environment is simpler to handle. The best option is to combine both: continuous-time setup for the dynamic system (physical objects are time-continuous), discrete-time setup for control (controllers are digital). Thus, we will operate with hybrid systems, which exhibits both continuous and discrete dynamic behavior. %, namely sample-and-hold.
%Thus, hybrid design for \blue{a} closed-loop setting is used as the most plausible for this work and all further configurations of algorithms are considered within a dynamic system of continuous-time in an environment via a sample-and-hold setting (S\&H).
Specifically, the following so-called sample-and-hold setting (S\&H), where the control actions are held constant during $\delta$-intervals, is considered:
%Sample-and-hold is a system design in which parameters are fixed (i.e. control actions are considered constant) for a certain sampling time interval $ \delta$, during which the system is controlled.
%Thus, an agent trajectory optimization is carried out in discrete-time steps and the observed system is the following:
% Thus, the actor and the critic are merged, and joint optimization is carried out in discrete-time samples.
% sampling time period
\begin{equation}
	\label{eq:sys-SH}
    \begin{aligned}
        & \mathcal D^+ x = f(x, u^\delta), \\
        & x_k := x(k \delta), \\
        & u^\delta(t) \equiv u_k = \kappa(x_k), t \in \left[ k\delta, (k + 1)\delta \right],
    \end{aligned}
\end{equation}
where $\mathcal{D}^{+}$ is a suitable differential operator, $x$ -- state, $u$ -- action, $\kappa$ -- policy, $t$ -- time, $k \in \N$, $\delta > 0$ -- discretization step. 
Notice that the agent's actions are to be optimized at discrete time steps, which makes the problem tractable, unlike in a pure continuous setting.
%Throughout the paper, a sampled control will be denoted as $u^\delta(t) \equiv u_k, t \in [ k \delta, (k+1) \delta ]$ or just $u^\delta$.
%The following short form will be used throughout: 
%\begin{equation}
%    \mathcal{D}^{+} x = f(x, u^{\delta}), \\
%\end{equation}
%where $u^\delta(t) := u_k, t \in [k \delta, (k+1) \delta ), k \in \N$.
% and also denote $x_k := x(k \delta).$
% where $u_{\delta}$ -- action on period $\delta$
%Since the discrete representation of the system is used, the so-called sampling effect is observed.
% Effects of sampling in controllers in general
%\textbf{Sampling effect.} 
%Sampling has a wide range of definitions in different science areas. 
%In this work, the reduction of a continuous-time signal to a discrete-time signal is called sampling.
%Some articles have shown that there is a dependence of the model performance on the sampling rate. 
% This effect on the trajectory tracking control performance in model-based control algorithms is demonstrated . 
% for the discrete one.
In the next sections, predictive control mechanisms are discussed in relation to the described S\&H setup.

% $\Delta t$ (

% The reduction of a continuous-time signal to a discrete-time signal is called sampling. In theory, an ideal sampler should produce samples equivalent to the instantaneous value of the continuous signal at the observed points. Of course, the ideal sample rate depends on system configuration and settings, but there are some intuitive general sampling rate selection patterns. The stabilization quality decreases with the increase of the sampling period (slow agent response). Wherein, sampling too fast (small sampling rate) can overload the processor, while sampling too slow (big sampling rate) can decrease controllers response time and increase the risk of destabilizing the controller.

% Too small sampling size will have a negative impact on controller performance. Bigger size of sampling will not necessarily provide better performance, but it is a safer direction to move if we have any doubts.

% More on time discretization in QL

% Some thoughts on it

%  We want to investigate these effects more
\section{Predictive control}
\label{prediction-controls}
% General discussion on prediction
\textbf{Prediction horizon effects.}
% \green{As it was mentioned earlier Model Predictive Control is widely used in industry. }
In general, control over a prediction horizon may help improve the agent's performance and aid system stabilization by using long-term information about the future states.
%provides useful long-term information about the future states of the system, which can be used to improve the agent stabilization, make it more robust and precise.
%The challenge is to reveal this useful information with a reasonable accuracy and consistency.
At the same time, in presence of model prediction error, the prediction horizon length and also the prediction step size are subject to careful tuning.
In general, the prediction horizon length refers to a period starting from the current time to a point until which control actions are to be optimized.
%Prediction horizon parameters include the prediction step size and the prediction horizon length.
The prediction step size can be described in the following example: if it is two times bigger than the sampling time then the state prediction is two times finer than the sampled control actions.
Evidently, higher prediction horizons lead to higher computational complexity.
%The performance of the system is improved by a larger prediction horizon, while it increases the computational cost.
(the total number of action sequences to be searched over increases exponentially with the horizon).
At the same time, too short horizon may lead to a failure to even stabilize the system.
Traditionally, MPC, which is discussed in the next section, is considered a standard predictive controller.
%Nonetheless, the prediction horizon cannot be too small because it affects both performance and stability.
\begin{algorithm*}
	\caption{roll-out Q-learning}
	\label{alg:nroll}
	\begin{algorithmic}
		\STATE {\bfseries Input:} System model, sampling time $\delta$, prediction step multiplier $s$, prediction horizon $N$
		\WHILE{$true$}
		\STATE Get state $x_k$
		\STATE Push the current state-action pair into the buffer (experience replay)
		\STATE Update critic: $ \vartheta_k := \argmin \limits_\vartheta J^c_k(\vartheta)$ (see eq. \ref{eq:critic-cost})
		\STATE Update actor: $ \{u_{i|k}\}_{i}^{N} := \min \limits_{\{u_{i|k}\}_{i}^{N}} J^a_{\text{RQL}} \left( x_k|\{u_{i|k}\}_{i}^{N}; \vartheta_k \right) = \sum \limits_{i=1}^{N-1} \gamma^{i-1} \rho(\hat x_{i|k}, u_{i|k}) + \hat{Q}(\hat x_{N|k}, u_{N|k}; \vartheta_k)$, where the state sequence $\{\hat x_{i|k}\}_i^N$ is predicted via \eg \eqref{eqn:Euler}
		\STATE Apply the first action from the sequence, namely, $u_{1|k}$, to the system
		\ENDWHILE
	\end{algorithmic}
\end{algorithm*}
% (every pred step size seconds)
% (time horizon $N_a$ for actor optimization)
\subsection{Model Predictive Control} \label{sec:MPC}

A fairly general optimal control in the S\&H setting can be formulated as follows:
%The considered setup is the following:
\begin{equation}
    \begin{aligned}
        \min_{\{u_{i}\}} \quad & J_{\text{OC}} \left(x_0|\{u_{i}\} \right) := \sum_{i=1}^{H} \gamma^{i-1} \rho \left( \hat x_{i|0}, u_i \right), \\
        \textrm{s.t.} \quad & \hat x_{2|i} = \Phi(s \delta, \hat x_{1|i}, u_i), \spc \hat x_{1|0} = x_0, \\
        & \mathcal D^+ x = f(x, u^\delta),  
    \end{aligned}
\end{equation}
where $\rho$ is the stage cost (a reward or utility in case of maximization), $J_{\text{OC}}$ is the accumulated stage cost, where $\gamma$ is the discounting factor, $H$ is the horizon length which can be finite ($H=N, N \in \N$) or infinite ($H=\infty$), $s$ is the prediction step size multiplier \ie the prediction step size is $s \delta$, $\Phi$ is a numerical integration scheme \ie $\hat x_{2|i} = \Phi(\delta, \hat x_{1|i}, u_i)$ is the predicted state emerging from $\hat x_{1|i}$ after the time $\delta$ and under the constant action $u_i$.
If $H = \infty$, $J_{\text{OC}}$ is also known as cost-to-go.
The simplest numerical integration scheme is the Euler one:
\begin{equation}
	\label{eqn:Euler}
	\hat x_{2|i} := \hat x_{1|i} + s \delta f(\hat x_{1|i}, u_i).
\end{equation}
In the following, the system dynamics $\mathcal D^+ x = f(x, u^\delta)$ are always meant, but omitted for brevity.
Based on $H$, two basic optimal control formalisms are generally known, namely, Euler-Lagrange and Hamilton-Jacobi-Bellman \cite{Primbs1999-opt-ctrl-CLF}.
Euler-Lagrange formalism possesses a ``local'' character -- it seeks a controller that optimizes the cost some $N$ steps ahead starting from the current state.
Hamilton-Jacobi-Bellman, in contrast, is ``global'' -- the goal here is to find a controller for cost optimization over an indefinite number of future steps. 
An infinite horizon may also be interpreted as an open horizon -- a situation, in which the user is unsure of an exact specification of the horizon.
In turn, a finite-horizon optimal control problem may be interpreted as a computationally tractable approximation of the infinite-horizon one.
%Finite-horizon case is an alternative to solve optimal control problems with reduced computational complexity.
MPC is the de facto scheme for finite-horizon optimal control problems \cite{Nikolaou01, Grune12, Mayne14}.
Here, the infinite horizon is cut at some finite time and an optimal solution is computed for the new fixed horizon at each time step.
Numerous modifications and a wide variety of techniques for guaranteeing closed-loop stability of MPC were developed \cite{Chen98, Scokaert99}.
In the S\&H setting, a simple unconstrained MPC setup can be written as:
% \begin{equation}
%     \begin{aligned}
%         \min_{\{u_{i}\}} \quad & \sum_{k=1}^{N} \rho \left(x_{k}, u_{k}\right), \\
%         \textrm{s.t.} \quad & \mathcal{D}^{+} x = f(x, u^{\delta}), \\
%         & \forall i=k, \ldots, k+N, \quad \\
%         & x_{i} \in \mathbb{S}, u_{i} \in \mathbb{A}, \\
%     \end{aligned}
% \end{equation}
\begin{equation}
    \begin{aligned}
        \min_{\{u_{i|k}\}_{i}^{N}} \quad & J_{\text{MPC}} \left( x_k| \{u_{i|k}\}_{i}^{N} \right) := \sum_{i=1}^{N} \gamma^{i-1} \rho \left(\hat x_{i|k}, u_{i|k}\right), \\
        \textrm{s.t.} \quad & \hat x_{i+1|k} = \Phi(s \delta, \hat x_{i|k}, u_{i|k}). \\
%        & y = h(x), \\ 
%        & \hat x_{i|k} \in \X, u_{i|k} \in \U, \\
    \end{aligned}
\end{equation}
%In a S\&H setting and in presence of prediction errors, $\X$ might need to be tightened to ensure that the actual state $x_{i|k}$ is in $\X$.
% $N_a$ - number of actor's steps
% This method is "straightforward”, so it evaluates the finite-horizon directly, without approximation.
When requiring constraint satisfaction of the kind $x_{i|k} \in \X, u_{i|k} \in \U$, MPC has the advantage of guaranteed safety \cite{Mayne14}.
Also, various schemes for stabilization guarantees are known \cite{Grune12, Mayne14}.
Evidently, a simple MPC controller is suboptimal, if the suboptimality is meant as the difference between the factual cost-to-go under the MPC controller and the optimized cost-to-go (the value function).
%compared to the globally optimal one due to the horizon cut. 
This is somewhat in contrast to the philosophy behind RL where an agent seeks to approximate the value function.
% Horizon cut -- sub-optimality. 
% MPC controller optimizes finite-horizon cost.
% RL approximates globally optimal controller.
% MPC -- less ambitious and not risky, RL -- more ambitious and more risky.
% in addition to the reduced computational complexity from infinite horizon length.
% Effects of horizon in MPC
In general, enlarging the horizon leads to suboptimality reduction \cite{Grune08}.
As mentioned above, careful tuning is required in general.
%This is since the increasing of the number of observations and, accordingly, \blue{obtaining} a more accurate estimate of the value function.
The next section discusses specifically fusion of MPC-elements with RL.

% From MPC to RL
% So, we want to investigate similar effects from control theory in RL.

\section{Fusion of RL and predictive controls}
\label{fusion}

% Ways to integrate prediction horizon into RL

Value iteration Q-learning (QL) actor-critic is chosen here as the basis for RL algorithms due to its convenience, although similar derivations could be done for the standard value and policy iteration. 
%As the basis for the fusion of RL and control, the Q-learning algorithm is picked although similar derivations could be made for the other standard RL algorithms as Value Iteration and Policy Iteration.
%Nominal Q-learning is a model-free RL algorithm that seeks to find the best action to take given the current state.
%Q-learning is realized as an iterative equation on so-called Q-functions or quality functions.
%In this algorithm, two main parts can be distinguished -- actor and critic.
%Critic task is the evaluation of an agent policy and computing suitable changes in the minimization objective of the subsequent action policy.
%The actor task is to calculate control actions through the minimization of the quality function.
% Also distinguish 2 approaches of iteration in RL tasks, Q-learning in particular: value iteration and policy iteration. Policy iteration -- random policy is selected and the corresponding value function is sought. Then a new improved policy is found based on the previous value function, and so on till the optimal policy. Value iteration -- value function is randomly selected, then a new improved value function in an iterative process is found until reaching the optimal value function. Then the optimal policy is derived from that optimal value function.
A basic online, model-free, value iteration, on-policy QL with a neural network critic reads:
% \begin{equation}
%     \begin{aligned}
%         & \kappa_i(x_k) := \argmin_{u} Q_i (x_k, u), \forall x_k, \\
%         & Q_{i+1}(x_k, u) := \rho(x_k, u) + Q_i(x_{k+1}, \kappa_i(x_{k+1})), \forall x_k, u, \\
%     \end{aligned}
% \end{equation}
\begin{equation}
    \begin{array}{lll}
        % \min_{\theta}
        u_k & := & \argmin \limits_u \hat Q(x_k, u; \vartheta_k), \\
        \vartheta_k & := & \argmin \limits_\vartheta \frac 1 2 \big( \hat Q(x_k, u_k; \vartheta) - \\
        & & \hat Q(x_{k-1}, u_k; \vartheta^-) - \rho(x_k, u_k) \big)^2, \\
%        \min_{\vartheta} \quad & J_{c} (x_{k+1}|\vartheta) =  \hat{Q}_{i+1}(x_k, u, \vartheta) -  \\
%        & Q_i \left(x_{k+1}, \kappa_i(x_{k+1}) \right) - \rho(x_k, u), \\
%        \textrm{where} \quad & Q(x, u) := \rho(x, u) + J(x_{+}^{u}), \\
%        \hat x_{k+1} & := & \Phi(\delta, x_k, u_k), \\
    \end{array}
\end{equation}
where $\vartheta$ is vector of the critic neural network weights to be optimized, $\vartheta^-$ is the vector of the weights from the previous time step, $\hat Q(\bullet, \bullet; \vartheta)$ -- Q-function approximation parameterized by $\vartheta$.
The latter approximation is effectively the temporal difference (TD) in the value iteration form.
It may be generalized to a custom size experience replay.
Let $e_{k}(\vartheta) = \vartheta \varphi( x_{k-1}, u_{k-1} ) - \gamma \vartheta^- \varphi(x_{k}, u_{k}) - \rho(x_{k-1}, u_{k-1})$ denote the temporal difference at time step $k$.
% \begin{equation}
%     \label{eq:TD}
%     \begin{aligned}
%         & \min_{\vartheta} J_{c} = \dfrac{1}{2} \sum_{k = 1}^{N_{c}} e_{k}^{2}, \\
%         \textrm{where} \quad & e_{k} = \vartheta_{k} \varphi( y_{k-1}, u_{k-1} ) - \\
%         & \gamma \vartheta_{k-1} \varphi(y_{k}, u_{k}) - \rho(y_{k-1}, u_{k-1}), \\
%         \textrm{and} \quad & \varphi (y, u) = \textrm{vec} \left(\Delta_{u} \left([y|u] \otimes [y|u] \right) \right), \\
%     \end{aligned}
% \end{equation}
Then, a more general critic cost function may be formulated as
\begin{equation}
	\label{eq:critic-cost}
	J^{c}_k(\vartheta) = \dfrac{1}{2} \sum_{i = k}^{k+M-1} e_{i}^{2}(\vartheta),
\end{equation}
where $M$ is the buffer size.
%Often this differential equation solution cannot be solved analytically, therefore the next state $x_{+}^{u}$ cannot be obtained directly.
%That's why different approximations are used to predict the next state.
%QL can be made purely data-driven and model free if $x_{k-1}, x_k$ are used in place of $x_k, \hat x_{k+1}$ in the above, respectively. 
%The next section presents the first variant of predictive QL to be addressed in the current study.
%As a combination of approaches of classical control theory and RL, the integration of the prediction horizon into RL via N-roll-out Q-learning and stacked Q-learning is considered in this work as the basis for the study on hyperparameter effects on RL.
%\subsection{roll-out Q-learning}

The \textit{roll-out QL} (RQL) considered here, given the MPC background of Section \ref{sec:MPC}, can be regarded as simply $N-1$ horizon MPC with a terminal cost being the Q-function approximation, namely, its actor reads: 
\begin{equation}
    \begin{aligned}
        \min_{\{u_{i|k}\}_{i}^{N}} & J^a_{\text{RQL}} \left( x_k|\{u_{i|k}\}_{i}^{N}; \vartheta_k \right) := \\ 
        & \sum_{i=1}^{N-1} \gamma^{i-1} \rho(\hat x_{i|k}, u_{i|k}) + \hat{Q}(\hat x_{N|k}, u_{N|k}; \vartheta_k), \\
        \textrm{s.t.} \quad & \hat x_{i+1|k} = \Phi(s \delta, \hat x_{i|k}, u_{i|k}). \\
    \end{aligned}
\end{equation}

%\subsection{Stacked Q-learning}
% Stacked setting is the modification when value functions (or Q-functions) are optimized not by a single value, but as a stack (sum) of them. 
% critic is modified so as to include not a single value function approximant, but rather a finite stack of them.
% In contrast to optimal control schemes with finite-horizon cost functions, where the stage costs r are stacked over a finite horizon – such as model predictive control (MPC) – the Q-functions are stacked in this case.  
% Since the Q-function represents infinite-horizon optimal control, the described control scheme may be roughly seen as a mixture of infinite- and finite-horizon optimal control.
% So, initial problem transformed to find an optimal input sequences. % do we need to add it?
% Concept idea of Stacked Q-learning is that through the time steps $k$ to $k+N$, the input sequences are computed with respect to the initial states.
% It merges ADP and MPC - stacked ADP (sADP). Aside from optimality over an infinite horizon, which is provided by the ADP part, the stabilization property of MPC under a terminal constraint is included as well. 
% need to formulate for value functions

\begin{algorithm*}[tb]
	\caption{Stacked Q-learning}
	\label{alg:sql}
	\begin{algorithmic}
		\STATE {\bfseries Input:} System model, sampling time $\delta$, prediction step multiplier $s$, prediction horizon $N$
		\WHILE{$true$}
		\STATE Get state $x_k$
		\STATE Push the current state-action pair into the buffer (experience replay)
		\STATE Update critic: $ \vartheta_k := \argmin \limits_\vartheta J^c_k(\vartheta)$ (see eq. \ref{eq:critic-cost})
		\STATE Update actor: $ \{u_{i|k}\}_{i}^{N} := \min \limits_{\{u_{i|k}\}_{i}^{N}} J^a_{\text{SQL}} \left( x_k|\{u_{i|k}\}_{i}^{N}; \vartheta_k \right) = \sum \limits_{i=1}^{N} \hat{Q}(\hat x_{i|k}, u_{i|k}; \vartheta_k)$, where the state sequence $\{\hat x_{i|k}\}_i^N$ is predicted via \eg \eqref{eqn:Euler}
		\STATE Apply the first action from the sequence, namely, $u_{1|k}$, to the system
		\ENDWHILE
	\end{algorithmic}
\end{algorithm*}

\noindent The \textit{stacked QL} (SQL)\cite{Osinenko2017stacked, Beckenbach18}, in turn, can be regarded as simply MPC with the stage cost substituted for Q-function approximation.
The justification for such a setup may be done via Lemma \ref{lem:stacked-QL}.
It says that the optimal policy from optimization of a stacked Q-function is essentially the same as the globally optimal one. 
First, denote the stacked Q-function as follows:
\begin{equation}
    \begin{aligned}
        \bar{Q} \left(x_k,  \{ u_{i|k}\}_{i}^N \right) := \sum_{i=1}^{N} Q(x_{i|k}, u_{i|k}). \\
        \label{eq:stacked-value-function}
    \end{aligned}
\end{equation}
With this notation at hand, proceed to the lemma.
%For the sake of completeness, below is a justification of equivalence of the stacked Q-function optimization to \blue{the} regular Q-function optimization.
 
% \begin{equation}
%     \bar{V}(x_{k}, \{ U_{i|k}\}_{N-1}) = \sum_{i=0}^{N-1} \sum_{j=k+i}^{\infty}  \rho(x_{j},u_{k+i}(x_{j}))
%     \label{eq:stacked-value-function}
% \end{equation}

% \begin{equation}
%     V(x_{k}, U_k) = \sum_{j=k}^{\infty} \rho(x_{j},u_j),
%     \label{eq:value-function}
% \end{equation} 

\begin{lemma}
\label{lem:stacked-QL}
% \begin{theorem}
% \begin{lemma}
For any $x_k$, it holds that
%For the time step $k$, let Q-function $Q(x,u)$ be defined by %\eqref{equiv:Gamma-V} and 
%\eqref{eq:value-function}
%and the stacked Q-function $\bar{Q}(x,u)$ by
%\eqref{eq:stacked-value-function}. Then
% \begin{align}
%     & \bar{Q}^{*} \left(x_k, \{ u_{i|k}\}_{N} \right) = \dfrac{1}{N} \sum_{i=0}^{N - 1} Q^{*}(x_{i|k}, u_{i|k}) , \\
%     & \quad \quad \forall k \in \N_{\geq 0},
%     \label{lemma:V*-leq-Gamma}
% \end{align}

\begin{equation}
	\label{lemma:V*-leq-Gamma}
    \begin{aligned}
        & \min_{\{ u_{i|k} \}_{i}^N} \bar{Q} \left(x_k, \{ u_{i|k}\}_{i}^N \right) = \sum_{i=1}^{N} \min_{ u_{i|k}}  Q(x_{i|k}, u_{i|k}). \\
%        & \quad \quad \forall k \in \N_{\geq 0},  
    \end{aligned}
\end{equation}
%\ie optimization of a stacked Q-function is equivalent to  
%This means, that the optimal policy coming from the usual QL is the same as the optimal policy obtained from the optimization of the stack of multiple Q-functions, so-called Stacked QL.
\end{lemma}
% \end{theorem}

\begin{proof}
    Let the optimal action sequence for the stacked Q-learning be denoted as:
    
    \begin{equation}
        \begin{aligned}
            & \{ \bar{u}_{i|k}^{*}\}_{i}^{N} := \argmin_{\{ u_{i|k}\}_{i}^N} \bar{Q} (x_k, \{ u_{i|k}\}_{i}^N) \\
            & = \sum_{i=1}^{N} Q(x_{i|k}, u_{i|k}), 
        \end{aligned}
    \end{equation}
    
    The optimal action sequence of the element-wise Q-function optimization reads:
    
    \begin{equation}
        \begin{aligned}
            & \{ u_{i|k}^{*}\}_{i}^{N} := \{\argmin_{u_{i|k}} Q (x_{i|k}, u_{i|k})\}_i.
        \end{aligned}
    \end{equation}
    
    Denote the corresponding optimal state sequences as $\{ \bar{x}_{i|k}^{*}\}_{i}^{N} := \{ \bar{x}_{k}^{*}, \bar{x}_{k+1}^{*}, \ldots, \bar{x}_{k + N - 1}^{*}\}$ for the stacked Q-learning and $\{ x_{i|k}^{*}\}_{i}^{N} := \{ x_{k}^{*}, x_{k+1}^{*}, \ldots, x_{k + N - 1}^{*}\} $ for the ordinary one.
    Notice that the initial state is the same: $\bar{x}_{k}^{*} = x_{k}^{*} = x_k $. By definition the optimal Q-function is equal to the value function $V$:
    \begin{align}
        V(x_{i|k}^{*}) = \min_{u} Q(x_{i|k}^{*}, u), \\
        V(\bar{x}_{i|k}^{*}) = \min_{u} Q(\bar{x}_{i|k}^{*}, u).
    \end{align}
    
    Denote 
    
   \begin{equation}
        \bar{V}(\bar{x}_{i|k}^{*}) := Q(\bar{x}_{i|k}^{*}, \bar{u}_{i|k}^{*}).
    \end{equation}
    
    By the principle of optimality it holds that
    
    \begin{equation}
        V(\bar{x}_{i|k}^{*}) \leq \bar{V}(\bar{x}_{i|k}^{*}).
    \end{equation}
    
    But since $x_{i|k}^{*}$ is the optimal state sequence, 
    
    \begin{equation}
        V(x_{i|k}^{*}) \leq V(\bar{x}_{i|k}^{*}).
    \end{equation}
    
    Thus, $V(x_{i|k}^{*}) \leq \bar{V}(\bar{x}_{i|k}^{*})$ and applying sum to the both parts yields
    
    \begin{equation}
        \sum_{i=1}^{N} V(x_{i|k}^{*}) \leq \sum_{i=1}^{N} \bar{V}(\bar{x}_{i|k}^{*}) = \min_{\{ u_{i|k} \}_{i}^N} \bar{Q} (x_{k}, \{u_{i|k}\}_{i}^N). 
    \end{equation}
    
    On the other hand, the minimum of the sum is not greater than the sum of minima:
    
    \begin{equation}
        \min_{\{ u_{i|k} \}_{i}^N} \bar{Q} (x_{k}, \{u_{i|k}\}_{i}^N) \leq \sum_{i=1}^{N} V(x_{i|k}^{*}).
    \end{equation}
    
    Then, as required,
    \begin{equation}
    	\begin{aligned}
	        & \min_{\{ u_{i|k} \}_{i}^N} \bar{Q} (x_{k}, \{u_{i|k}\}_{i}^N) = \sum_{i=1}^{N} V(x_{i|k}^{*}) = \\ 
	        & \sum_{i=1}^{N} \min_{ u_{i|k}}  Q(x_{i|k}, u_{i|k}).
        \end{aligned}
    \end{equation}
    
\end{proof}

Since in practice, Q-functions cannot always be computed exactly, a temporal-difference-based critic can be employed to compute approximate stacked Q-function using the temporal difference method \eqref{eq:critic-cost}.
This can be done \eg using neural networks.
% setting is the following (actor's structure)
Finally, the stacked QL actor reads:
\begin{equation}
    \begin{aligned}
        \min_{\{u_{i|k}\}_{i}^{N}} \quad & J^a_{\text{SQL}} \left( x_k|\{u_{i|k}\}_{i}^{N}; \vartheta_k \right) = \sum_{i=1}^{N} \hat{Q}(\hat x_{i|k}, u_{i|k}; \vartheta_k), \\
        \textrm{s.t.} \quad & \hat x_{i+1|k} = \Phi(s \delta, \hat x_{i|k}, u_{i|k}).
    \end{aligned}
\end{equation}

% \begin{algorithm}[tb]
% 	\red{clean up}
%   \caption{Stacked Q-learning}
%   \label{alg:sql}
%     \begin{algorithmic}
%       \STATE {\bfseries Input:} $x$, $u$
%       \WHILE{$t < \tau$}
%       \STATE Sample with time-step $\delta$
%       \STATE System output prediction: $\mathcal{D}^{+} x = f(x, u^{\delta}), y = h(x)$
%       \STATE Critic update: $ \min_{\vartheta} J_{c} = \dfrac{1}{2} \sum_{k = 1}^{N_{c}} e_{k}^{2}$
%       \STATE Actor update: $ \min_{\{u_{i|k}\}_{i}^{N}} J_{a} \left( y_k|\{u_{i|k}\}_{N} \right) =\frac{1}{N} \sum_{i=0}^{N-1} \hat Q(x_{i|k}, u_{i|k})$
%       \STATE System update (next output prediction):
%       \STATE $x^{+} = x + \delta \, f(x, u), y^{+} = h(x)$ 
%       \ENDWHILE
%     \end{algorithmic}
% \end{algorithm}
The main features of the described algorithms for the clarity are given in Table~\ref{table:algms}.
\begin{table*}
	\caption{Algorithms intuition}
	\label{table:algms}
	\centering
	\begin{tabular}{lll}
		\toprule
		Name & Scheme & Description \\
		\midrule
		\textit{MPC} & Baseline & Finite sum of stage costs without a terminal cost \\
		\textit{RQL} & $\textrm{MPC} + Q_N$ & Finite sum of stage costs with a Q-function as the terminal cost \\
		\textit{SQL} & $\textrm{MPC} \wedge \textrm{QL}: r \la Q$ & Finite sum of Q-functions \\
		\bottomrule
	\end{tabular}
\end{table*}
%%%%%%%%%%%%%%%%%%%%%%%%%%%%%%%%%%%%%%%%%%%%%%%%%%%%%%%%%%%%%%%%%%%%%%%%%%
%%%%%%%%%%%%%%%%%%%%%%%%%%%%%%%%%%%%%%%%%%%%%%%%%%%%%%%%%%%%%%%%%%%%%%%%%%
%%%%%%%%%%%%%%%%%%%%%%%%%%%%%%%%%%%%%%%%%%%%%%%%%%%%%%%%%%%%%%%%%%%%%%%%%%
\section{Numerical experiments}

\label{experiments}
% The main evidence of sensitivity to parameters such as time sampling, horizon length, and prediction step size is the simulation performed for various parameter values.
All three described setups, namely, MPC, roll-out QL and stacked QL were studied in numerical experiments with wheeled robot parking.
In every experiment, a set of hyperparameters consisting of the horizon length $N$, prediction step size multiplier $s$, and the sampling time $\delta$, was fixed.
An experiment consisted of 30 runs, 600 s long each.
The robot started at a position on a 5 m circle around the origin, turned away from the latter.
The goal was to park the robot at the origin while achieving a desired orientation.
The parking was considered successful if the robot entered a 50 cm circle around the origin with a 5 deg tolerance in angle.
The accumulated stage cost, as well as the successful parking count, were considered the performance metrics.
The next section describes the system dynamics in detail.

\subsection{Environment} 
% Mobile robot

As the dynamic system, the three-wheel robot with dynamical pushing force and steering torque (a.k.a. ENDI -- extended non-holonomic double integrator) was considered. 
% \cite{Abbasi17}. 

In Cartesian coordinates system description is the following:
\begin{equation}
    \label{system}
    \begin{aligned}
        & \dot{x} = v \cos \alpha, \\
        & \dot{y} = v \sin \alpha, \\
        & \dot{\alpha} = \omega, \\
        & \dot{v} = \left( \dfrac{1}{m} F \right), \\
        & \dot{\omega} = \left( \dfrac{1}{I} M \right). \\
    \end{aligned}
\end{equation}

where $x$ -- $x$-coordinate [m], $y$ -- $y$-coordinate [m], $\alpha$ -- turning angle [rad], $v$ -- velocity [m/s], $\omega$ -- angular velocity [rad/s], $F$ -- pushing force [N], $M$ -- steering torque [Nm], $m$ -- robot mass [kg], $I$ -- robot moment of inertia around vertical axis [kg $m^2$] ($m = 10$, $I = 1$).

\subsection{Simulation}

Implementation of MPC, roll-out Q-learning, and stacked Q-learning were done in a custom python framework \texttt{rcognita}\footnote{\url{https://github.com/AIDynamicAction/rcognita}}, developed specifically for hybrid simulation of RL agents (Fig.~\ref{simulation}). \\

\begin{figure}[ht]
    \begin{center}
    \centerline{\includegraphics[width=\columnwidth]{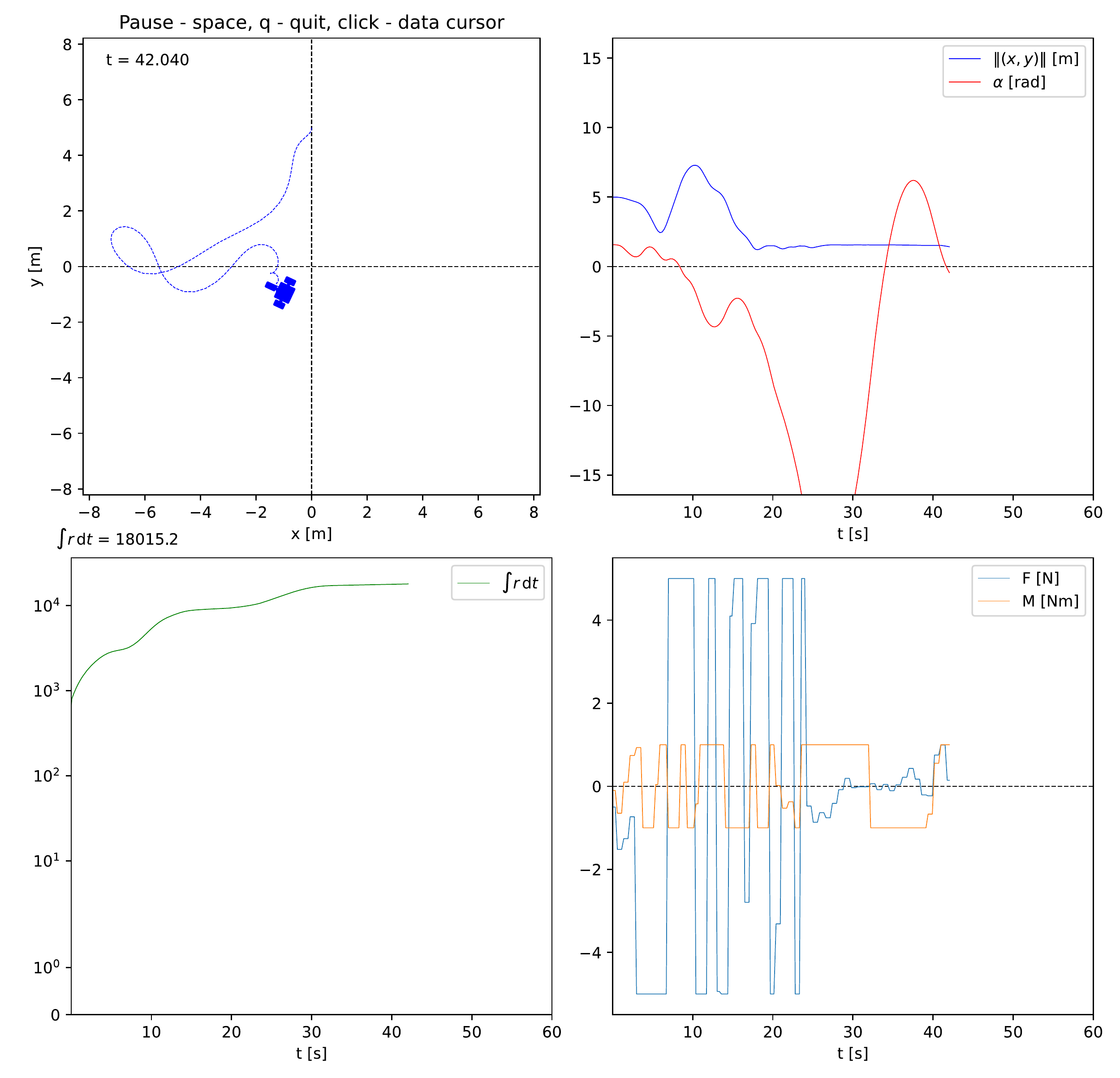}}
    \caption{Graphical output of the \texttt{rcognita} Python package.}
    \label{simulation}
    \end{center}
%\vskip -0.2in
\end{figure}

%\texttt{rcognita}

The stage cost was considered in the following quadratic form:
\begin{equation}
    \begin{aligned}
        & \rho = \chi^\top R \chi, \\
    \end{aligned}
\end{equation}

where $\chi = [y, u]$, $R$ diagonal, positive-definite.
The critic structure was also chosen quadratic as follows:
\begin{equation}
    \begin{aligned}
        & \hat{Q}(x,u; \vartheta) := \vartheta \varphi^\top (x, u), \\
        & \varphi (x, u) := \textrm{vec} \left(\Delta_{u} \left([x|u] \otimes [x|u] \right) \right), \\
    \end{aligned}
\end{equation}

where $\vartheta$ -- critic weights, $\varphi$ -- critic activation function,
% (previous outputs and actions are taken from the buffers $U$, $Y$ (buffer size is chose equal to 200),
$\Delta_{u}$ -- operator of taking the upper triangular matrix, $\textrm{vec}$ -- vector-to-matrix transformation operation, $[x|u]$ -- stack of vectors $x$ and $u$,  $\otimes$ -- Kronecker product. 
%The choice of the structure of the critic feature vector is quadratic without mixed terms.
% $[y^2, u^2]$.

% weights are updated according to the so-called temporal differences to achieve subsequent lower costs.

% Variables sizes: critic weights $\vartheta$ size is $[1, L]$, control action $u$ size is $[1, l]$, system output $y$ size is $[1, p]$, size of buffer $U$ of previous actions: $[N_{c}, l]$, size of buffer $Y$ of previous outputs: $[N_{c}, p]$, $n$ - state size, critic stack size (not greater than buffer Size) $N_{critic} = 50$, state dimension $dimState = 5$, action dimension $dimInput = 2$, output dimension $dimOutput = 5$.

% Sampling time $\Delta t$ is scalar: $\Delta t > 0$, prediction step size multiple of $\Delta t$, $\gamma = 1$. 

% Machinery for hyperparameters

%For each hyperparameter choice, a simulation was done via an experimental run for each of 30 generated initial points at the same distance equal to 5 meters and termination time $\tau$ = 600 sec.
%For demonstration of  effects were used 5 sampling values in range $[0.001, 0.02]$, 5  values of prediction step size in range $[2, 6]$ and 10 prediction horizon lengths in range $[1, 10]$.
%In addition, bar charts for the number of successful parkings are provided. 
%Next, experimental results based on the described framework and corresponding discussion are provided. 
The experimental results are presented below.

\begin{figure}[ht]
    \begin{center}
    \centerline{\includegraphics[width=\columnwidth]{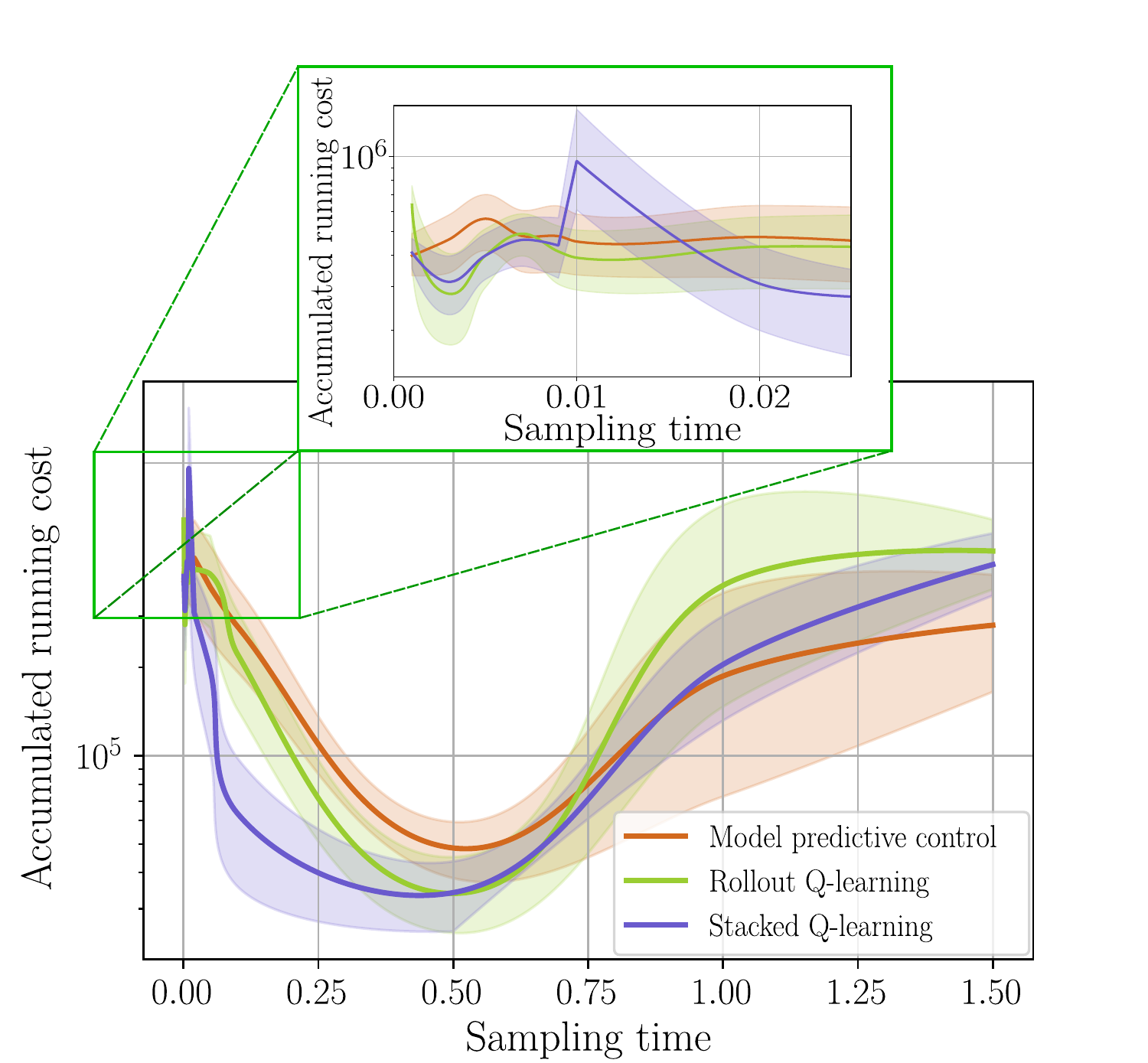}}
    % \caption{Demonstration of sampling effects on different actor's mode with different $\Delta t$. }
    \caption{Relationship between the accumulated stage cost and the sampling time. Solid line -- average over 30 observations, shaded area -- 95 \% confidence level.}
    \label{fig:sampling_rate}
    \end{center}
%\vskip -0.2in
\end{figure}

\begin{figure}[ht]
%\vskip 0.2in
    \begin{center}
    \centerline{\includegraphics[width=\columnwidth]{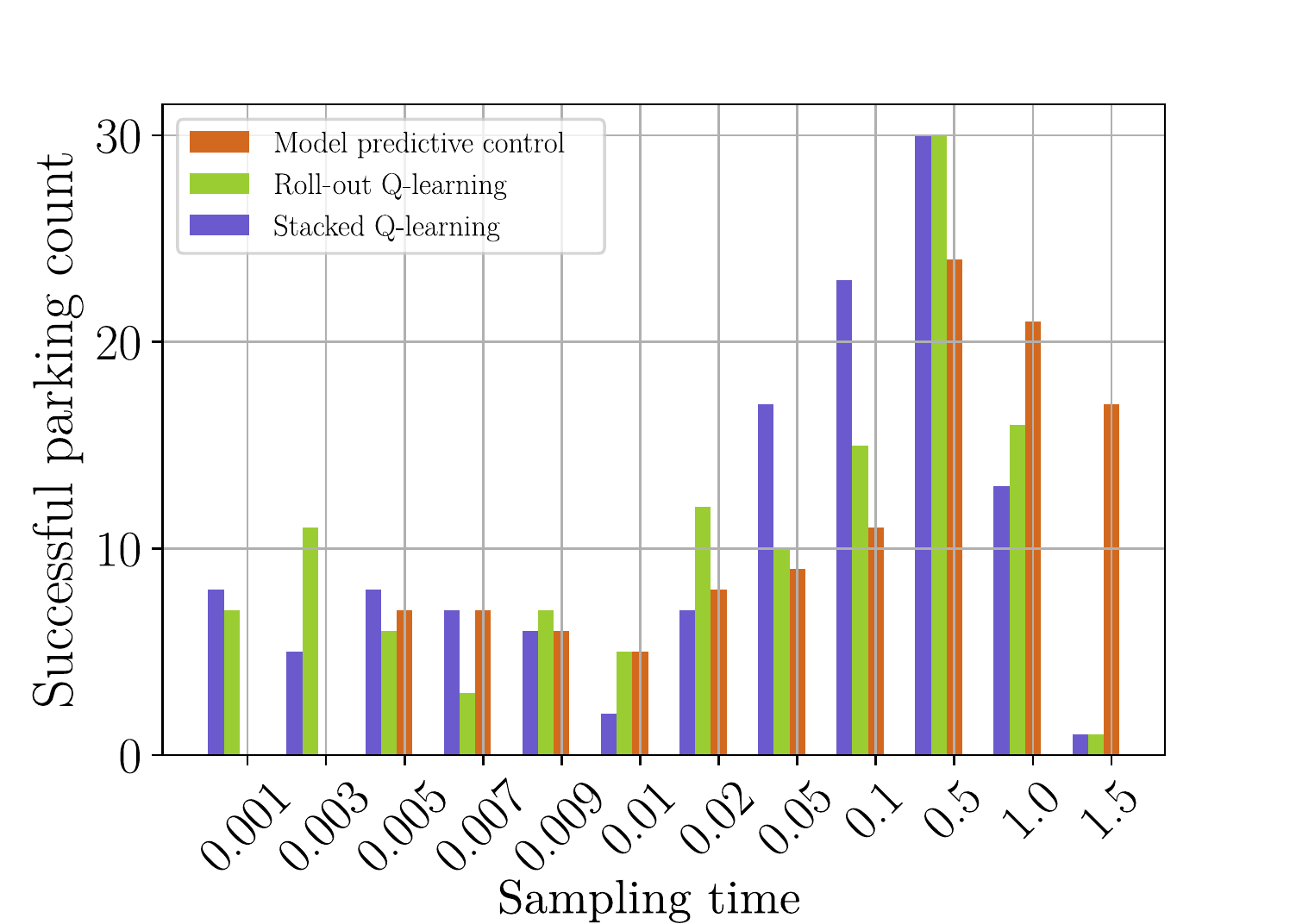}}
    % \caption{Demonstration of sampling effects on different actor's mode with different $\Delta t$. }
    \caption{Successful parking count depending on the sampling time.}
    \label{fig:sampling_rate_park}
    \end{center}
%\vskip -0.2in
\end{figure}

\begin{figure}[ht]
%\vskip 0.2in
    \begin{center}
    \centerline{\includegraphics[width=\columnwidth]{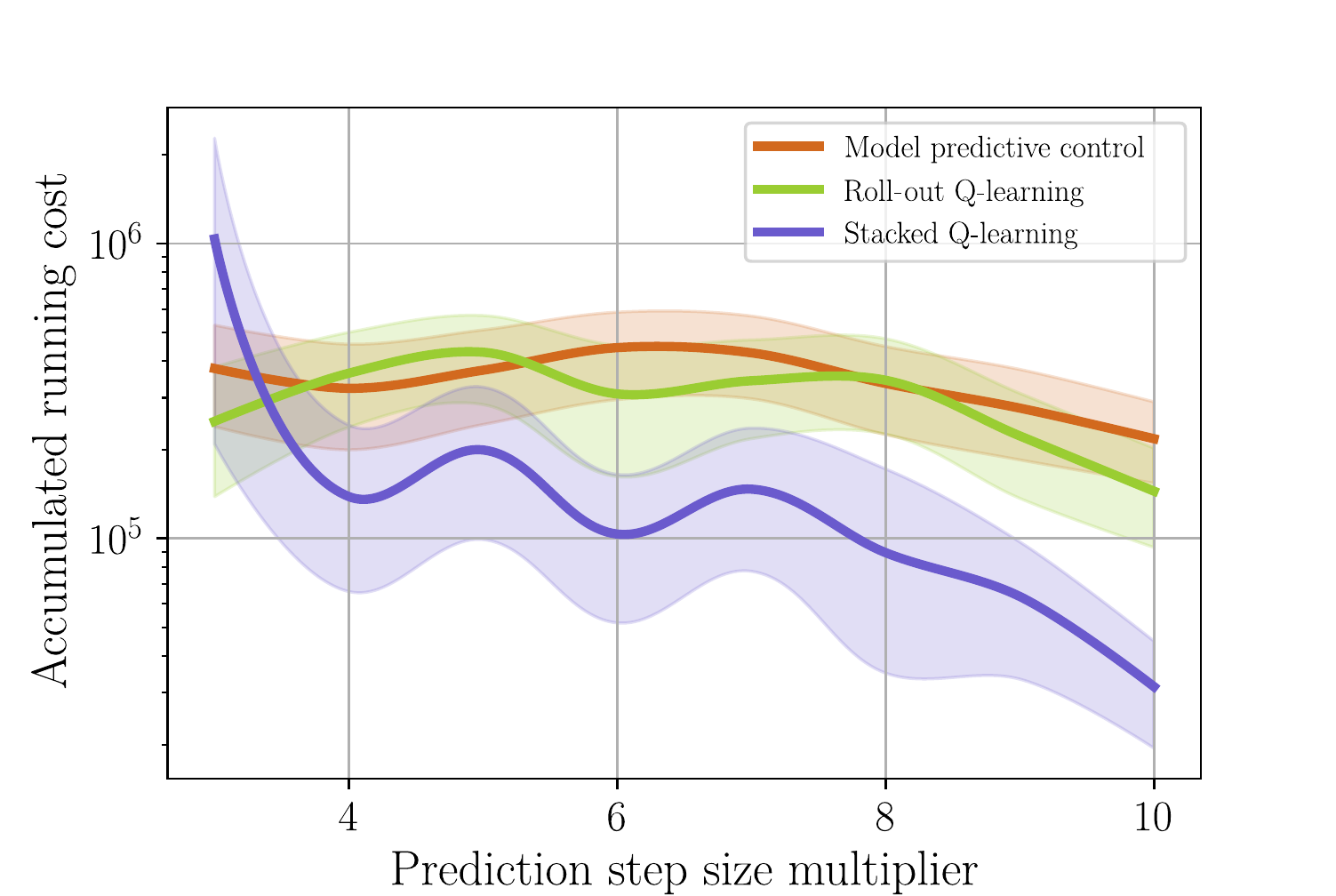}}
    \caption{Relationship between the accumulated stage cost and the prediction step size multiplier. Solid line -- average over 30 observations, shaded area -- 95 \% confidence level.}
    \label{fig:pred_step_size}
    \end{center}
%\vskip -0.2in
\end{figure}

\begin{figure}[ht]
%\vskip 0.2in
    \begin{center}
    \centerline{\includegraphics[width=\columnwidth]{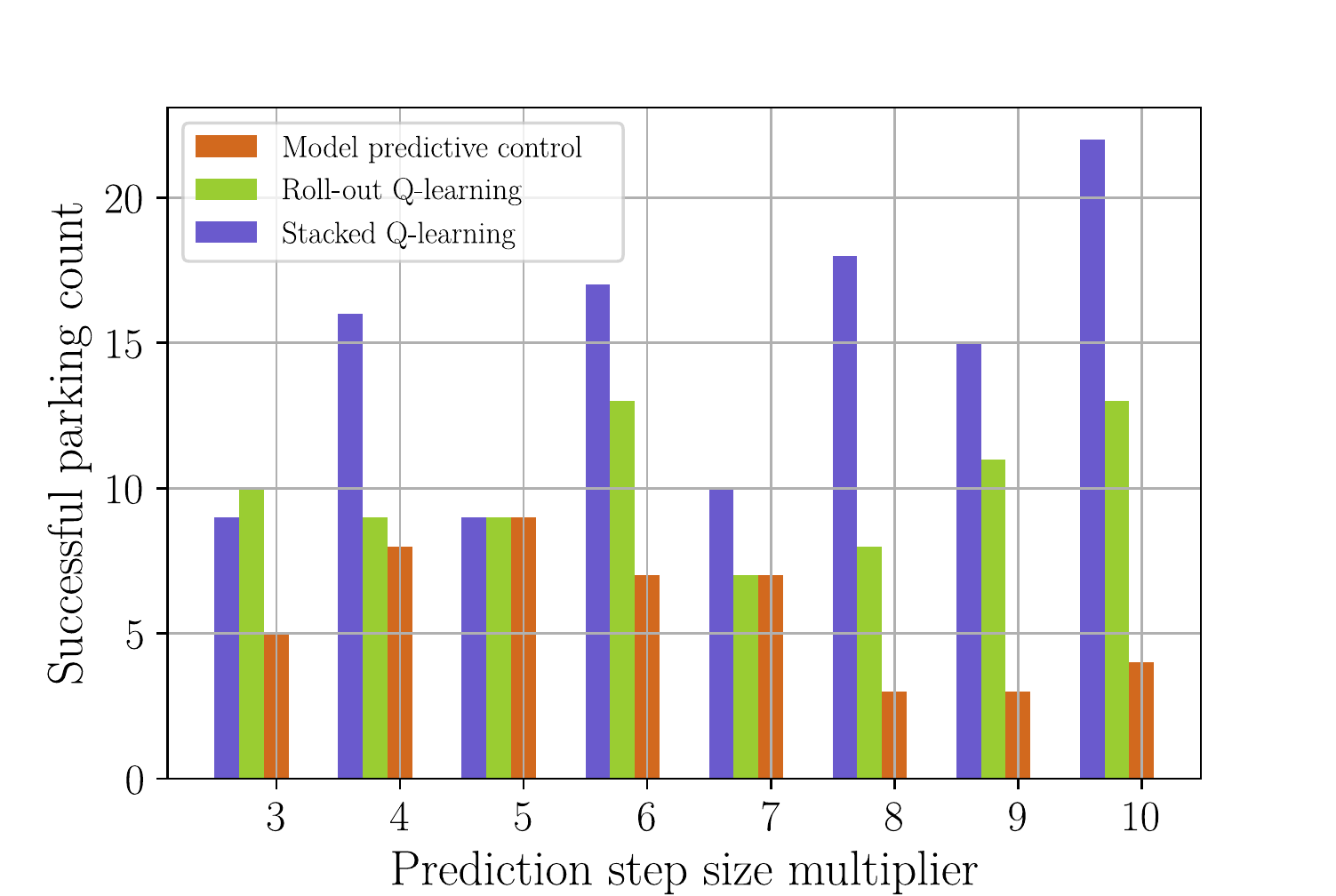}}
    \caption{Successful parking count depending on the prediction step size.}
    \label{fig:pred_step_park}
    \end{center}
%\vskip -0.2in
\end{figure}

% \begin{figure}[ht]
% \vskip 0.2in
%     \begin{center}
%     \centerline{\includesvg[width=\columnwidth]{gfx/predstepsize_results.svg}}
%     \caption{Relationship between stabilization time and prediction step size.}
%     \label{pred_step_size}
%     \end{center}
% \vskip -0.2in
% \end{figure}

\begin{figure}[ht]
%\vskip 0.2in
    \begin{center}
    \centerline{\includegraphics[width=\columnwidth]{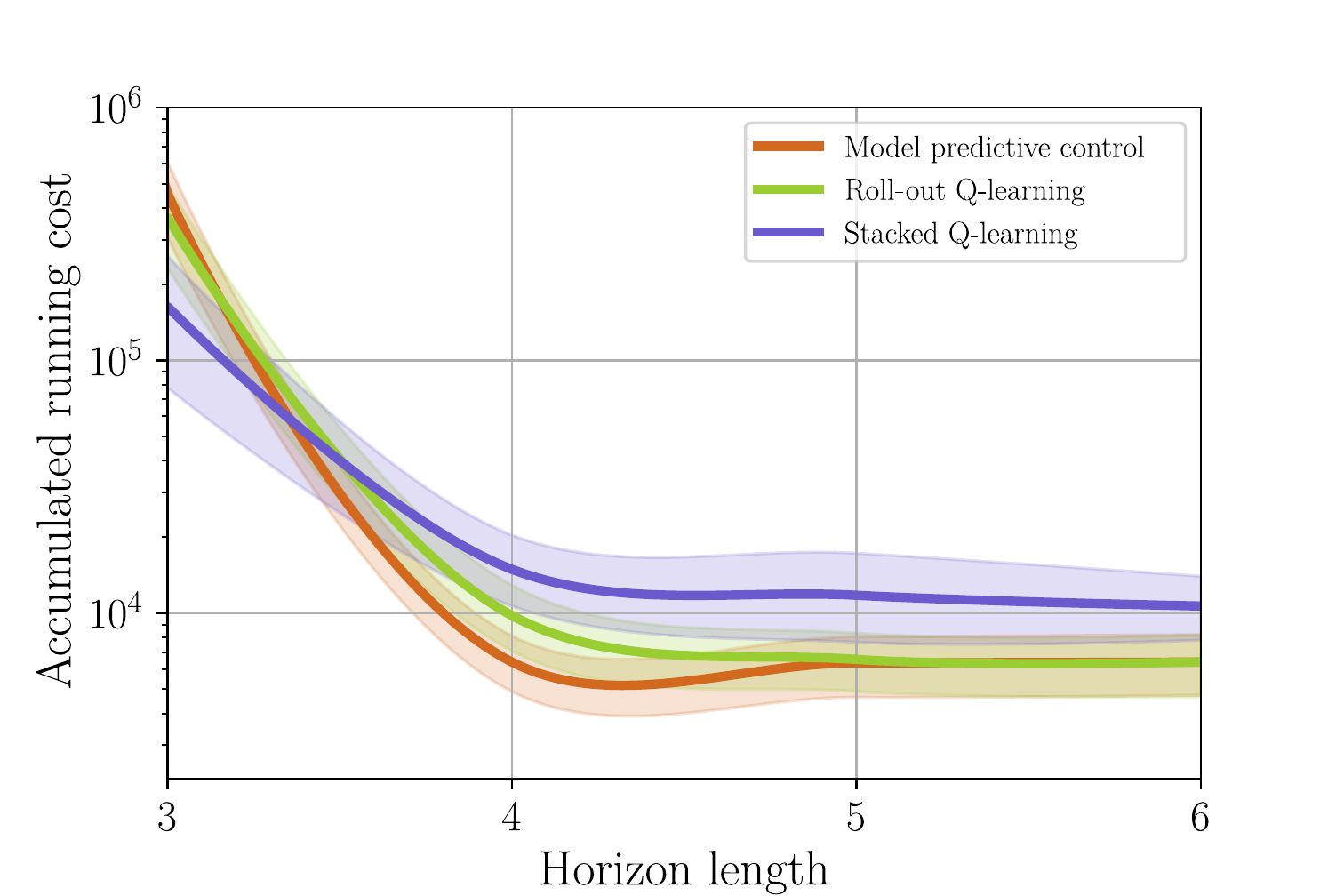}}
    \caption{Relationship between the accumulated stage cost and the prediction horizon length. Solid line -- average over 30 observations, shaded area -- 95 \% confidence level.}
    \label{fig:pred_horizon_size}
    \end{center}
%\vskip -0.2in
\end{figure}

\begin{figure}[ht]
%\vskip 0.2in
    \begin{center}
    \centerline{\includegraphics[width=\columnwidth]{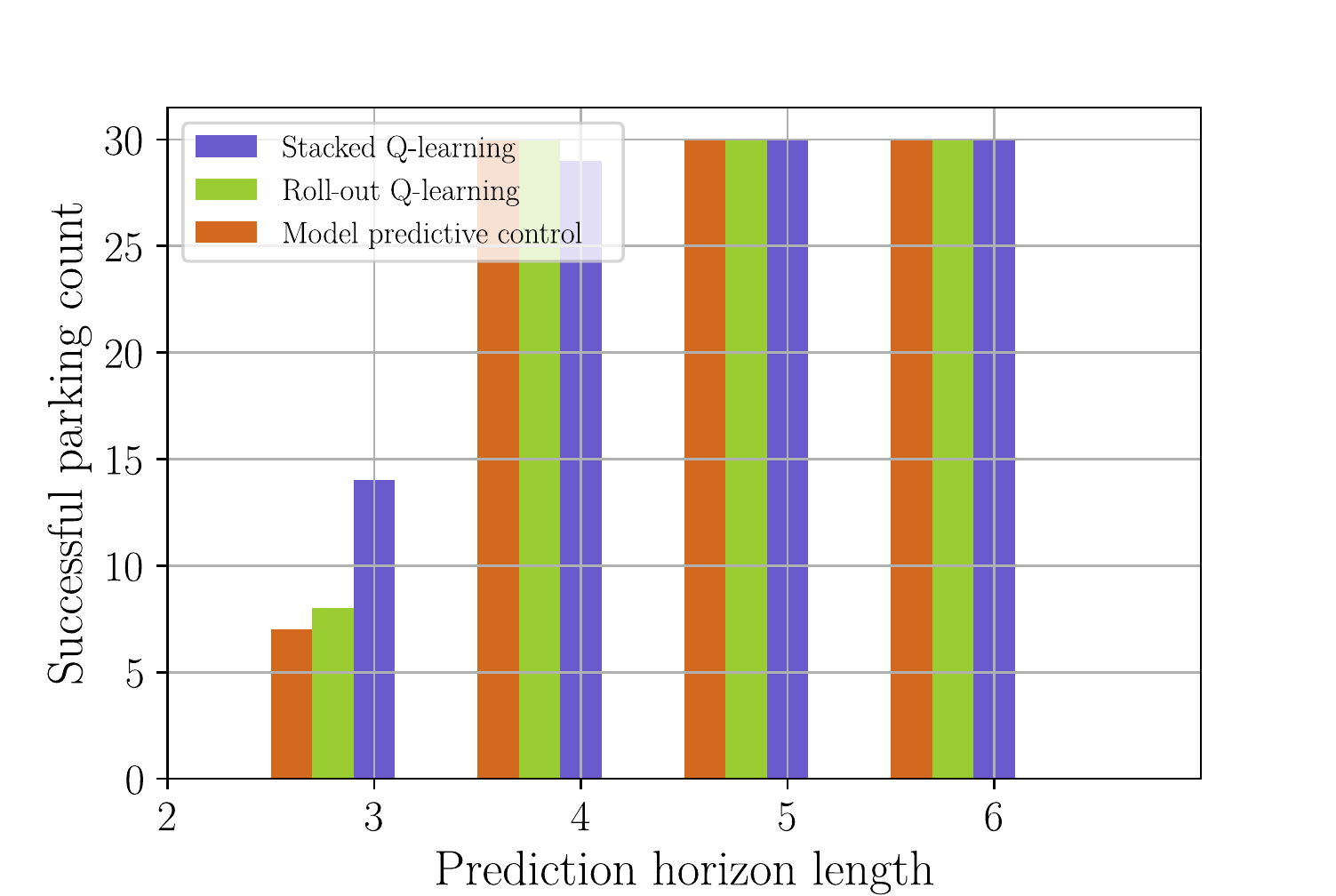}}
    \caption{Successful parking count tuning the prediction horizon length.}
    \label{fig:pred_hor_park}
    \end{center}
%\vskip -0.2in
\end{figure}

\section{Results and discussion}
\label{sec:discussion}
% Effects of sampling: observations

% The main results of the research are presented on the following plots Fig.~\ref{fig:sampling_rate},
% Fig.~\ref{pred_step_size},
% Fig.~\ref{pred_step_park},
% Fig.~\ref{pred_horizon_size},
% Fig.~\ref{pred_hor_park}.

% Each of them demonstrates the dependence described above for each of the algorithms.

It was observed that too low sampling time led to somewhat higher cost which can be explained by a too narrow-sighted controller.
Increasing the sampling time remedies that issue, but only up to a certain point where prediction inaccuracies start to dominate, whence the cost grows again (see Fig. \ref{fig:sampling_rate}).
A similar tendency can be observed in terms of successful parking count (see Fig. \ref{fig:sampling_rate_park}).
As the prediction step size was increased, it was observed that both the accumulated stage cost and successful parking count slightly improved.
This may explained by an effective horizon enlargement (though retaining the resolution).
Again, one should be aware of growing prediction errors while increasing the prediction step size.
As far as the horizon length itself is concerned, there was a clear tendency of performance improvement with higher $N$.
At the horizon length of 5, all the controllers succeeded to park the robot in all trials.

In terms of the algorithm comparison, some interesting phenomena could be noticed.
First of all, both the roll-out and stacked QL generally outperformed MPC at shorter sampling times and horizons.
Fig. \ref{fig:pred_step_park} in turns shows far superior successful parking compared to MPC.
These observations support the idea that integration of learning elements into classical controller \eg in the form of RL, is beneficial.
In theory, the Q-function captures the performance of the agent over an infinite horizon.
It seems logical that approximating the Q-function sufficiently well may yield better control actions than optimizing a plain sum of stage costs.
The following hint could be made: \textit{predictive RL is more beneficial than MPC at shorter horizons}.
This is because as the horizon length grows, predictive RL becomes indistinguishable from MPC (see Fig. \ref{fig:pred_horizon_size}), while both simply approach the globally optimal controller.
One should be aware of the computational complexity though.

Roughly, it can be described as follows.
Denote the MPC complexity (in terms of optimizing $J_{\text{MPC}}$) by $\mathcal O \left(\Psi_{\text{MPC}}(N) \right)$ for some function $\Psi_{\text{MPC}}$ of the horizon length (in particular, the complexity is exponential in $N$).
Then, the complexities of the roll-out and stacked QL are both $\mathcal O(\Psi_{\text{MPC}}(N)) + \mathcal O \left(\Gamma_{\textrm{QL}}(M, n_{\varphi}) \right)$ where $\Gamma_{\textrm{QL}}$ describes the complexity of the critic update (in terms of optimizing $J_c$), $M$ is the experience replay size and $n_{\varphi}$ is the number of critic weights.

A final note should be made about the difference in performance between the roll-out and stacked QL.
Remarkably, the latter significantly outperformed the former at a short horizon ($N$=3), both in terms of the accumulated stage cost and successful parking count.
This may be explained in a similar manner as above, when comparing RL with MPC.
Namely, learning elements of RL are more beneficial at shorter horizons.
Notice that the roll-out QL ``retains'' more from MPC than its stacked counterpart.
At longer horizons, such a structure of the roll-out QL becomes more beneficial than the stacked QL.
Notice also that nominally both QL methods have the same complexity.
However, taking into account better performance of the stacked QL at shorter horizons, the practical complexity of the stacked QL may even be considered lower than that of the roll-out variant.
That is, the stacked QL may achieve a comparable performance to the roll-out QL with a longer horizon.

\section{Conclusion}
\label{conclusion}	

As learning-based control becomes ever more attractive, it faces ever more challenges in industry, where, traditionally, such controllers as MPC are recognized due to their formal guarantees.
Reinforcement learning slowly transitions from playgrounds like videogames into more challenging environments.
On this path, it seems unavoidable that some of the well-established classical machinery, such as predictive control, can be made use of in RL.
This is supported, in particular, by the attractive trend of fusion of MPC and RL.
The current study was generally dedicated to this topic and considered a particular, yet fairly popular, control problem of parking of a mobile robot by MPC and RL.
The influence of the prediction- and sampling-related hyperparameters, namely, the prediction horizon step and length sizes, and the sampling time, was investigated.
It was generally observed that RL-based controllers appeared more efficient than MPC at shorter horizon lengths, where the learning-based elements dominated.
Predictive RL, like the herein studied stacked Q-learning, may be considered a viable solution in terms of fusion of classical controllers and RL agents.

\bibliographystyle{IEEEtran}
\bibliography{bib/ML, bib/MPC, bib/opt-ctrl, bib/Osinenko, bib/dynam, bib/survey, bib/parameters, bib/nonsmooth-analysis, bib/optimization, bib/ENDI}

\end{document}